\documentclass{birkjour}
 \newtheorem{thm}{Theorem}[section]
 \newtheorem{cor}[thm]{Corollary}
 
 \newtheorem{prop}[thm]{Proposition}
 \theoremstyle{definition}
 \newtheorem{defn}[thm]{Definition}
 \theoremstyle{remark}

 \numberwithin{equation}{section}

\begin{document}

%
%
%
%
%
%
%
%
%

\title[The Relation Between Automorphism Group and Isometry Group]
 {The Relation Between Automorphism Group and Isometry Group of Randers Lie Groups}

\author[H. R. Salimi Moghaddam]{Hamid Reza Salimi Moghaddam}

\address{%
Department of Mathematics, Faculty of  Sciences, University of Isfahan, Isfahan,81746-73441-Iran.}

\email{salimi.moghaddam@gmail.com and hr.salimi@sci.ui.ac.ir\\
Received: \\
Revised: }

\thanks{This work was supported by the research grant from Shahrood
university of technology.}



\date{October 26, 2010}

\begin{abstract}
In this paper we consider simply connected Lie groups equipped with left invariant Randers metrics which arise from
left invariant Riemannian metrics and left invariant vector fields. Then we study the intersection between automorphism and isometry groups of these spaces. Finally it has shown that for any left invariant vector field, in a special case, the Lie group admits a left invariant Randers metric such that this intersection is a maximal compact subgroup of the group of automorphisms with respect to which the considered vector field is invariant.
\end{abstract}

\maketitle
\section{Introduction}\label{Intro}
Randers metrics are interesting Finsler metrics which have found many applications in theoretical physics
and biology (for some applications of Finsler metrics see \cite{AnInMa, As, BaChSh}.). On the other hand, since
Randers metrics arise from Riemannian metrics and vector fields (or 1-forms),
working on them are simpler than Finsler metrics in general case. Because of the construction of Randers
metrics, it seems that many properties obtained in Riemannian case can be extended to this case with some
conditions.\\
An attractive class of structures in differential geometry is the family of invariant structures on Lie groups
with respect to the action of Lie groups on themselves. Invariant Riemannian metrics on Lie groups belong to this
family. These metrics have been studied by many mathematicians because the algebraic properties of a Lie group can impress on the geometric properties of it. They have found many significant results and examples on these spaces (for example see \cite{KoNo,Mi,No}.). Recently, many studies have been accomplished to extend these properties to Finsler spaces. On the other hand the algebraic structure of Lie groups can help us to find simple formulas for some complicated quantities related to Finsler spaces such as flag curvature (for example see \cite{DeHo1, DeHo2, EsSa1,EsSa2, Sa1, Sa2, Sa3, Sa4, Sa5}.).\\
In this article we study the intersection between automorphism and isometry groups of some Randers spaces. We consider left invariant Randers metrics which arise from left invariant Riemannian metrics and left invariant vector fields on
simply connected Lie groups.

\section{Preliminaries}\label{Prelim}

Let $G$ be a simply connected Lie group with Lie algebra
$\frak{g}$. Then we can define the following continuous
homomorphism:
\begin{equation}\label{T}
    \left\{%
\begin{array}{ll}
    T:Aut(G)\longrightarrow Aut(\frak{g}), \\
    \phi\longrightarrow T(\phi)=(d\phi)_e:T_eG=\frak{g}\longrightarrow\frak{g}.\\
\end{array}%
\right.
\end{equation}
Since $G$ is simply connected therefore $T$ is an isomorphism of
Lie groups (see \cite{Eb}). Also we have $K$ is a maximal compact
subgroup of $Aut(G)$ if and only if $T(K)$ is a maximal compact
subgroup of $Aut(\frak{g})$.\\

A Finsler metric on a manifold $M$ is a non-negative function
$F:TM\longrightarrow\Bbb{R}$ with the following properties:
\begin{enumerate}
    \item $F$ is smooth on the slit tangent bundle
    $TM^0:=TM\setminus\{0\}$,
    \item $F(x,\lambda Y)=\lambda F(x,Y)$ for any $x\in M$,
    $Y\in T_xM$ and $\lambda >0$,
    \item the $n\times n$ Hessian matrix $[g_{ij}]=[\frac{1}{2}\frac{\partial^2 F^2}{\partial y^i\partial
    y^j}]$ is positive definite at every point $(x,Y)\in TM^0$.
\end{enumerate}
A special type of a Finsler metrics is a Randers metric which has
been introduced by G. Randers in 1941 \cite{Ra}. Randers metrics
are constructed by using Riemannian metrics and vector fields (1-forms).\\
Let $g$ and $X$ be a Riemannian metric and a vector field on a
manifold $M$ respectively such that $\|X\|=\sqrt{g(X,X)}<1$. Then
a Randers metric $F$ can be defined by $g$ and $X$ as follows:
\begin{eqnarray}\label{Randers}
     F(x,Y)=\sqrt{g(x)(Y,Y)}+g(x)(X(x),Y), \ \ \ \ \forall x\in M,
     Y\in T_xM.\label{IRM}
\end{eqnarray}
In this paper, we denote an inner product on $\frak{g}$ and also the
corresponding left invariant Riemannian metric on $G$ by  $< , >$.\\

\section{Automorphism and Isometry Groups}

Suppose that $\phi\in Aut(G)$, then for any $g\in G$ we have:
\begin{equation}
    \phi\circ L_g=L_{\phi(g)}\circ\phi.
\end{equation}
Now let $X\in\frak{g}$ be a left invariant vector field such that
$<X,X><1$. We define a Randers metric on $G$ by using the
Riemannian metric $< , >$ and the left invariant vector field $X$ as
follows:
\begin{equation}\label{F}
    F(g,Y_g)=\sqrt{<Y_g,Y_g>}+<X_g,Y_g> \forall g\in G, Y_g\in
    T_gG.
\end{equation}
Obviously $F$ is left invariant as
\begin{eqnarray}
  F(L_gh,dL_gY_h) &=& \sqrt{<dL_gY_h,dL_gY_h>}+<X_{L_gh},dL_gY_h> \nonumber\\
  &=&\sqrt{<Y_h,Y_h>}+<dL_gX_h,dL_gY_h> \\
  &=&\sqrt{<Y_h,Y_h>}+<X_h,Y_h>=F(h,Y_h).\nonumber
\end{eqnarray}

\begin{defn}
For an isomorphism $\phi\in Aut(G)$ and a vector field $X$ on $G$
we say that $X$ is $\phi$-invariant if $d\phi\circ X=X\circ\phi$.
\end{defn}

\begin{prop}\label{Iso}
Suppose that $F$ is as above. Let $I(G,<,>)$ and $I(G,F)$ denote
the isometry groups of Riemannian manifold $(G,<,>)$ and Randers
manifold $(G,F)$ respectively. Suppose that $\phi\in I(G,<,>)$,
then $\phi\in I(G,F)$ if and only if $X$ is $\phi$-invariant.
\end{prop}

\begin{proof} Let $X$ be $\phi$-invariant therefore for any $g\in
G$ we have $X_{\phi(g)}=d\phi_gX_g$. Then
\begin{eqnarray}
  F(\phi(g),d\phi_gY_g) &=& \sqrt{<d\phi_gY_g,d\phi_gY_g>}+<X_{\phi(g)},d\phi_gY_g>\nonumber\\
  &=& \sqrt{<Y_g,Y_g>}+<d\phi_gX_g,d\phi_gY_g>\\
  &=&\sqrt{<Y_g,Y_g>}+<X_g,Y_g>=F(g,Y_g).\nonumber
\end{eqnarray}
Conversely let $\phi\in I(G,F)$, therefore we have:
\begin{eqnarray}
   \sqrt{<Y_g,Y_g>}+<X_g,Y_g>&=&F(g,Y_g)=F(\phi(g),d\phi_gY_g)\\
   &=&
   \sqrt{<d\phi_gY_g,d\phi_gY_g>}+<X_{\phi(g)},d\phi_gY_g>\nonumber,
\end{eqnarray}
which shows that
\begin{equation}
    <d\phi_gX_g-X_{\phi(g)},Y_g>=0.
\end{equation}
Therefore we have $d\phi_gX_g=X_{\phi(g)}$.
\end{proof}
\begin{cor}
Let $X$ be $\phi$-invariant for any $\phi\in I(G,<,>)$. Then
$I(G,<,>)=I(G,F)$.
\end{cor}
\begin{proof}
Since for every $\phi\in I(G,<,>)$ $X$ is $\phi$-invariant the
above proposition says that $I(G,<,>)\subset I(G,F)$. On the other
hand proposition 1.3 of \cite{DeHo2} says that $I(G,F)\subset
I(G,<,>)$, therefore the proof is completed.
\end{proof}
\begin{defn}
Suppose that a Lie group $G$ acts from the left (right) on a
manifold $M$, in other words let $M$ be a $G$-manifold. We say
that a metric (Finsler or Riemann) on $M$ is absolutely left
(right) invariant if $I(M)=G$, where $I(M)$ is the isometry group
of $M$ with respect to the metric.
\end{defn}

\begin{cor}
Let $<,>$ be an absolutely left invariant Riemannian metric on a
Lie group $G$ and $X$ be a left invariant vector field on $G$ such
that $<X,X><1$. Then $I(G,<,>)=I(G,F)=G$, where $F$ is defined by
$<,>$ and $X$.
\end{cor}

\begin{prop}
Let $<,>$ be any left invariant Riemannian metric on a Lie group
$H$. Also suppose that $F$ is a left invariant Randers metric
introduced by $<,>$ and a vector field $X$. Then $X$ is a left
invariant vector field.
\end{prop}

\begin{proof}
For any $g,h\in G$ we have
\begin{eqnarray}
   F(h,Y_h)&=&\sqrt{<Y_h,Y_h>}+<X_h,Y_h>\nonumber\\
   &=&\sqrt{<dL_gY_h,dL_gY_h>}+<X_{L_gh},dL_gY_h>\\
   &=&F(L_gh,dL_gY_h),\nonumber
\end{eqnarray}
which shows that
\begin{equation}
   <X_h-dL_{g^{-1}}X_{L_gh},Y_h>=0,
\end{equation}
therefore $X$ is left invariant.
\end{proof}
\begin{defn}
Let $X$ be a left invariant vector field on the simply connected Lie group $G$.
Then we define the following subgroups of $Aut(G)$ and $Aut(\frak{g})$:
\begin{eqnarray}
  Aut_X(G)&=&\{\phi\in Aut(G)| \mbox{$X$ is $\phi$-invariant}\} \\
  Aut_X(\frak{g})&=&\{d\phi_e\in Aut(\frak{g})|
  \mbox{$X_e$ is an eigenvector of $d\phi_e$} \\
  && \ \ \ \ \ \mbox{with respect to eigenvalue $1$. i.e. $d\phi_eX_e=X_e$}\}\nonumber.
\end{eqnarray}
Where $e$ is the unit element of $G$.
\end{defn}

\begin{prop}
Suppose that $G$ is a Lie group with Lie algebra $\frak{g}$. Also let $<,>$ be a left
invariant Riemannian metric on $G$ and $X$ be a left invariant vector field such that $<X,X><1$.
If $\phi\in Aut(G)$ such that $T(\phi)=d\phi_e\in Aut_X(\frak{g})$ is a linear isometry of $\frak{g}$
with respect to the inner product $<,>$, then $\phi\in I(G,F)$, where $F$ is the Randers
metric defined by $<,>$ and $X$.
\end{prop}

\begin{proof}
From the fact that $\phi$ is a homomorphism, it follows immediately that for any $g\in G$
$\phi\circ L_g=L_{\phi(g)}\circ\phi$. Therefore
\begin{eqnarray*}
  d\phi_gX_g&=&d(\phi\circ L_g)_eX_e\\
  &=&d(L_{\phi(g)}\circ\phi)_eX_e\\
  &=&(dL_{\phi(g)})_e(d\phi_eX_e)\\
  &=&(dL_{\phi(g)})_eX_e=X_{\phi(g)}.
\end{eqnarray*}
So we have $\phi\in Aut_X(G)$. On the other hand we have
\begin{eqnarray*}
  <(d\phi)_gY_g,(d\phi)_gZ_g> &=& <(d\phi)_g(dL_g)_eY_e,(d\phi)_g(dL_g)_eZ_e>\\
  &=&<(dL_{\phi(g)})_e(d\phi)_eY_e,(dL_{\phi(g)})_e(d\phi)_eZ_e> \\
  &=&<Y_e,Z_e> = <Y_g,Z_g>,
\end{eqnarray*}
therefore $\phi\in I(G,<,>)$. Now we have
\begin{eqnarray*}
  F(\phi(g),(d\phi)_gY_g) &=& \sqrt{<(d\phi)_gY_g,(d\phi)_gY_g>}+<X_{\phi(g)},(d\phi)_gY_g> \\
   &=& \sqrt{<Y_g,Y_g>}+<(d\phi)_gX_g,(d\phi)_gY_g>\\
   &=&\sqrt{<Y_g,Y_g>}+<X_g,Y_g>\\
   &=&F(g,Y_g),
\end{eqnarray*}
hence $\phi\in I(G,F)$.
\end{proof}

\begin{cor}
The isomorphism $T:Aut(G)\longrightarrow Aut(\frak{g})$ maps \\
$K=Aut_X(G)\bigcap I(G,F)$ isomorphically onto $K'=Aut_X(\frak{g})\bigcap O(\frak{g})$, where $O(\frak{g})$ denotes the orthogonal group of $\{\frak{g},<,>\}$.
\end{cor}

\begin{cor}
If $Aut_X(\frak{g})\bigcap I(G,<,>)=Aut(\frak{g})\bigcap I(G,<,>)$ and
$K'=Aut_X(\frak{g})\bigcap I(G,<,>)$ then both $K'$ and $K=T^{-1}(K')$ are compact Lie groups because $K'$ is a closed subgroup of the compact
Lie group $O(\frak{g})$.
\end{cor}

\begin{prop}\label{F-exist}
Let $X$ be a left invariant vector field on $G$ and $K$ be any compact subgroup of $Aut_X(G)$. Then there exists
a left invariant Randers metric $F$ on $G$ such that $K\subset Aut_X(G)\bigcap I(G,F)$, where $F$ is defined by $X$
and a left invariant Riemannian metric $<,>$ on $G$.
\end{prop}

\begin{proof}
Since $K$ is a compact subgroup of $Aut_X(G)\subset Aut(G)$ there exists a left invariant Riemannian
metric $<,>_0$ on $G$ such that $K\subset Aut(G)\bigcap I(G,<,>_0)$ (see \cite{Eb}.). Therefore
$K\subset Aut_X(G)\bigcap I(G,<,>_0)$. Now let $<,>=\frac{1}{N}<,>_0$ for some $N\in\Bbb{N}$ such that $<X,X><1$.
Obviously we have $I(G,<,>)=I(G,<,>_0)$. Now let $F$ be the Randers metric defined by $X$ and the Riemannian metric
$<,>$. Proposition \ref{Iso} says that $Aut_X(G)\bigcap I(G,<,>)=Aut_X(G)\bigcap I(G,F)$ which
completes the proof.
\end{proof}

\begin{prop}
Let $G$ be a simply connected Lie group, $X$ a left invariant vector field on $G$ and
$Aut_X(\frak{g})\bigcap I(G,<,>)=Aut(\frak{g})\bigcap I(G,<,>)$. Then there exists a left invariant
Randers metric $F$ on $G$ such that \\
$K=Aut_X(G)\bigcap I(G,F)$ is a maximal compact subgroup of $Aut_X(G)$.
\end{prop}

\begin{proof}
Let $K$ be a maximal compact subgroup of $Aut_X(G)$. By proposition \ref{F-exist} there exists a left invariant Randers
metric on $G$ such that $K\subset Aut_X(G)\bigcap I(G,F)$. Equality holds by the maximality of $K$.
\end{proof}


\begin{thebibliography}{1}
\bibitem{AnInMa} P. L. Antonelli, R. S. Ingarden, M. Matsumoto, \textit{The Theory of Sprays
and Finsler Spaces with Applications in Physics and Biology},
Kluwer Academic Publishers, (1993).
\bibitem{As} G. S. Asanov, \textit{Finsler Geometry, Relativity and Gauge Theories},
D. Reidel Publishing Company, (1985).
\bibitem{BaChSh} D. Bao, S. S. Chern and Z. Shen, \textit{An Introduction to Riemann-Finsler
Geometry}, (Berlin: Springer) (2000).

\bibitem{DeHo1} S. Deng and Z. Hou, \textit{Invariant Finsler Metrics on Homogeneous Manifolds},
J. Phys. A: Math. Gen. \textbf{37} (2004), 8245--8253.
\bibitem{DeHo2} S. Deng and Z. Hou, \textit{Invariant Randers Metrics on Homogeneous Riemannian
Manifolds}, J. Phys. A: Math. Gen. \textbf{37} (2004), 4353--4360.
\bibitem{Eb} P. Eberlein, \textit{Left Invariant Geometry of Lie Groups}, Cubo, \textbf{6, no.1}, (2004), 427-510.
\bibitem{EsSa1} E. Esrafilian and H. R. Salimi Moghaddam, \textit{Flag Curvature of Invariant Randers
Metrics on Homogeneous Manifolds}, J. Phys. A: Math. Gen.
\textbf{39} (2006) 3319--3324.
\bibitem{EsSa2} E. Esrafilian and H. R. Salimi Moghaddam, \textit{Induced Invariant
Finsler Metrics on Quotient Groups}, Balkan Journal of Geometry
and Its Applications, Vol. \textbf{11}, No. 1 (2006) 73--79.
\bibitem{KoNo} S. Kobayashi and K. Nomizu, \textit{Foundations of Differential Geometry Vol. 2},
Interscience Publishers, John Wiley\& Sons, (1969).
\bibitem{Mi} J. Milnor, \textit{Curvatures of Left Invariant Metrics on Lie Groups}, Advances in Mathematics
\textbf{21} (1976) 293--329.
\bibitem{No} K. Nomizu, \textit{Invariant Affine Connections on Homogeneous Spaces}, Am. J. Math.
\textbf{76} (1954) 33--65.
\bibitem{Ra} G. Randers, \textit{On an Asymmetrical Metric in the Four-Space of General Relativity},
Phys. Rev. \textbf{59}(1941), 195--199.
\bibitem{Sa1} H. R. Salimi Moghaddam, \textit{On the flag curvature of invariant Randers
metrics}, Math. Phys. Anal. Geom.  \textbf{11} (2008) 1--9.
\bibitem{Sa2} H. R. Salimi Moghaddam, \textit{Flag curvature of invariant $(\alpha,\beta)$-metrics
of type $\frac{(\alpha+\beta)^2}{\alpha}$}, J. Phys. A: Math.
Theor.  \textbf{41} (2008).
\bibitem{Sa3} H. R. Salimi Moghaddam, \textit{Some Berwald spaces of non-positive flag curvature},
Journal of geometry and physics,  \textbf{59} (2009) 969--975.
\bibitem{Sa4} H. R. Salimi Moghaddam, \textit{Randers Metrics of Berwald type on 4-dimensional hypercomplex Lie groups},
J. Phys. A: Math. Theor.  \textbf{42} (2009).
\bibitem{Sa5} H. R. Salimi Moghaddam, \textit{On the Geometry of Some Para-Hypercomplex Lie Groups},
Archivum Mathematicom BRNO, \textbf{45} (2009) 159--170.


\end{thebibliography}
\end{document}